\documentclass[twocolumn,journal]{IEEEtran}

\usepackage{cite,comment}
\usepackage{amsthm}
\usepackage{mathrsfs}
\usepackage{graphicx}
\usepackage{graphics}
\usepackage{amssymb}
\usepackage{amsmath,mathtools}
\usepackage{color}
\usepackage{xspace}
\usepackage{algpseudocode}
\usepackage{bbm}
\usepackage{comment}
\usepackage{algorithmicx}
\usepackage[caption=false]{subfig}
\usepackage{psfrag}
\usepackage{url}
\usepackage{float}
\usepackage{caption}
\usepackage{soul}
\usepackage{algpseudocode}
\usepackage{algorithm,algorithmicx}
\usepackage{rotating}
\usepackage{chngcntr}
\usepackage{apptools}
\AtAppendix{\counterwithin{thm}{section}}
\usepackage{graphicx}
\usepackage{graphics}
\usepackage{amssymb}
\usepackage{amsmath}
\usepackage{mathtools}
\usepackage{color}
\usepackage{xspace}
\usepackage{algpseudocode}
\usepackage{bbm}
\usepackage{comment}
\usepackage{algorithmicx}
\usepackage{subfig}
\usepackage{psfrag}
\usepackage{tikz}
\usetikzlibrary{shapes,arrows}
\usetikzlibrary{positioning}

\usepackage{rotating}
\usepackage{chngcntr}
\usepackage{apptools}
\usepackage{listings}
\AtAppendix{\counterwithin{thm}{section}}
\usepackage{graphicx}
\usepackage{graphics}
\usepackage{amssymb}
\usepackage{amsmath}
\usepackage{color}
\usepackage{xspace}
\usepackage{algpseudocode}
\usepackage{bbm}
\usepackage{comment}
\usepackage{algorithmicx}
\usepackage{subfig}
\usepackage{psfrag}
\usepackage{tikz}
\usetikzlibrary{shapes,arrows}
\usetikzlibrary{positioning}

\usepackage{enumitem}
\usepackage{graphicx}               
\usepackage{amsmath,amssymb}

\usepackage{graphics}
\usepackage{amssymb}
\usepackage{amsmath}
\usepackage{color}
\usepackage{xspace}
\usepackage{algpseudocode}
\usepackage{bbm}
\usepackage{comment}
\usepackage{algorithmicx}
\usepackage{subfig}
\usepackage{psfrag}
\usepackage{soul}
\usepackage{tikz}
\usetikzlibrary{petri}
\usetikzlibrary{shapes,arrows}
\usetikzlibrary{positioning}

\usepackage{cancel}

\tikzstyle{block}=[draw opacity=0.7,line width=1.4cm]

\newcommand{\oprocendsymbol}{\hbox{$\bullet$}}
\newcommand{\oprocend}{\relax\ifmmode\else\unskip\hfill\fi\oprocendsymbol}

\newcommand{\VV}{\mathcal{V}}
\newcommand{\EE}{\mathcal{E}}
\newcommand{\GG}{\mathcal{G}}

\newcommand{\lL}{\vect{\mathsf{L}}}

\newcommand{\vect}[1]{\boldsymbol{\mathbf{#1}}}
\newcommand{\vectsf}[1]{\boldsymbol{\mathbf{\mathsf{#1}}}}

\newcommand{\Diag}[1]{\operatorname{Diag}(#1)}

\newtheorem{assumption}{Assumption}[section]

\newtheorem{theorem}{Theorem}[section]

\newtheorem{lemma}{Lemma}[section]

\DeclareMathAlphabet{\mathpzc}{OT1}{pzc}{m}{it}

\title{\LARGE \bf
Distributed Unconstrained Optimization with \\ Time-varying Cost~Functions
}
\author{Amir-Salar Esteki and Solmaz S. Kia \emph{Senior Member, IEEE} \vspace{1em}
        \thanks{The authors are with the Department of Mechanical and Aerospace Engineering, University of California Irvine, Irvine, CA 92697,
        {\tt\small \{aesteki,solmaz\}@uci.edu}. This work was supported by NSF award ECCS-1653838.}%
}

\begin{document}

\maketitle
\thispagestyle{empty}
\pagestyle{empty}


\begin{abstract}
In this paper, we propose a novel solution for the distributed unconstrained optimization problem where the total cost is the summation of time-varying local cost functions of a group networked agents. The objective is to track the optimal trajectory that minimizes the total cost at each time instant. Our approach consists of a two-stage dynamics, where the first one samples the first and second derivatives of the local costs periodically to construct an estimate of the descent direction towards the optimal trajectory, and the second one uses this estimate and a consensus term to drive local states towards the time-varying solution while reaching consensus. The first part is carried out by the implementation of a weighted average consensus algorithm in the discrete-time framework and the second part is performed with a continuous-time dynamics. Using the Lyapunov stability analysis, an upper bound on the gradient of the total cost is obtained which is asymptotically reached. This bound is characterized by the properties of the local costs. To demonstrate the performance of the proposed method, a numerical example is conducted that studies tuning the algorithm's parameters and their effects on the convergence of local states to the optimal trajectory.

\end{abstract}

\section{INTRODUCTION}
In this paper, we consider the distributed time-varying unconstrained optimization problem, where a group of $N$ agents want to track the optimal solution that minimizes a total cost which is the summation of local time-varying costs. In recent years, the use of distributed optimization problems have attracted many applications \cite{chang2014multi, droge2014continuous, nedic2014distributed, gharesifard2013distributed}, e.g., sensor networks, smart grids, robotics~\cite{verscheure2009time} and \cite{ardeshiri2011convex}, and learning systems~\cite{koppel2015task}, due to the surge of multi-agent systems. One example is the time-varying distributed linear regression problem where the agents learn a model that best fits an ever changing stream of data, or see e.g., \cite{zhao1993training, feng1992time, ye2015distributed} for related works. In robotics, also, a group of connected robots aim to localize a moving object that can be cast as an optimization problem where local costs vary over time. To solve the time-varying optimization problem at hand, each agent needs to solve the following problem
\begin{align}\label{eq::problem}
    \vect{x}^{\star}(t):=\textup{argmin}_{\vect{x}\in\mathbb{R}^n}\frac{1}{N}\sum_{i=1}^{N}f^i(\vect{x},t), \quad t\geq0,
\end{align}
where $f^i$ represents agent $i$'s local cost and $\vect{x}\in\mathbb{R}^{n}$ is the decision variable. In case that there is a supervisor or a trusted third-party with the knowledge of all local $f^i$'s who can distribute the optimal trajectory to the agents online, a central solution can be implemented. Some prediction-correction-based algorithms both in continuous-time and discrete-time have been proposed as the central solution~\cite{simonetto2016class} to track the optimal trajectory $\vect{x}^{\star}$. By incorporating the second derivative of the costs (also called Hessian of the cost), i.e., $\nabla_{\vect{x}\vect{x}}f^i(\vect{x},t)$, this method asymptotically converges to the exact optimal trajectory.

Conversely, in many cases, due to privacy concerns or the distribution of data among multiple agents, employing a central solver is not feasible. Such settings require a distributed solver where agents are only allowed to communicate with their neighbors. Therefore, distributed algorithms have been proposed in the literature which take into account the limitations of a fully decentralized network. Some works consider network topology as the time-variant part of the problem \cite{rogozin2021accelerated, reisizadeh2022distributed, li2018accelerated} and some other, discussed below, consider the costs to vary over time. Authors of~\cite{rahili2016distributed} have addressed this problem by suggesting an algorithm where exact convergence is achieved in finite time by using the signum function, which limits the solution to the continuous-time setting. Moreover, agents need to compute the inverse of the Hessian locally at each time instant, which requires an $\mathcal{O}(n^3)$ computational complexity. In other works \cite{ling2013decentralized, chen2021distributed, maros2017admm}, using the alternating direction method of multipliers (ADMM), the authors propose a solution in the discrete-time framework for time-varying optimization problems. The approach used in \cite{maros2017admm} is inspired by the ADMM methods applied in static optimization problems, where asymptotic convergence is achieved if the objective function varies sufficiently low over time. Some other works, e.g. \cite{sun2017distributed} and \cite{ye2015distributed}, consider only optimization problems with time-varying quadratic costs, due to their popularity in applications such as economic dispatch. However, this limitation disallows them to be implemented for other classes of strongly-convex cost functions. Resource allocation problems are also a great part of economic dispatch that can be cast as constrained time-varying optimization problems; see \cite{esteki2022distributed, simonetto2018dual, fazlyab2017prediction, bai2018distributed, wang2020distributed, wang2022distributed} for related works.

In this paper, we address the distributed optimization problem with time-varying local costs in the continuous-time framework. Each agent communicates only with their neighbors and shares local information to solve this problem. By incorporating a weighted average consensus algorithm, agents track the descent direction towards the optimal trajectory and drive their local states to the time-varying optimal solution. This technique allows the agents to asymptotically converge to a neighborhood of the solution without using the signum function which prevents discrete-time implementations and also reduces the computational complexity to $\mathcal{O}(n^2)$. A Lyapunov stability analysis is conducted to prove convergence for strongly-convex and lipschitz-continuous local cost functions. In the numerical example section, we show the performance of the proposed method with different values of algorithm parameters.

\medskip
\emph{Notations:} We follow~\cite{bullo2009distributed} for graph theoretic terminologies. The interaction topology of $N$ in-network agents is modeled by the undirected connected graph $\mathcal{G}(\VV,\EE,\vect{A})$ where $\VV$ is the node set, $\EE\subset\VV\times \VV$ is the edge set and $\vect{A}=[\vectsf{a}_{ij}]$ is the adjacency matrix defined such that $\vectsf{a}_{ij}>0$, if $(i,j)\in\mathcal{E}$, otherwise $\vectsf{a}_{ij}=0$. A graph is undirected if $\vectsf{a}_{ij}=\vectsf{a}_{ji}$ for all $i,j\in\VV$. Moreover, a graph is connected if there is a directed path from every node to every other node. The degree of each node $i\in\VV$ is $\mathsf{d}^i=\sum_{j=1}^N \vectsf{a}_{ij}$ and the Laplacian matrix of a graph $\GG$ is  $\lL\!=\!\Diag{\mathsf{d}^1,\cdots, \mathsf{d}^N}-\vect{A}$. Furthermore, For a connected graph, we denote the eigenvalues of $\lL$ by $\lambda_1,\cdots,\lambda_N$, where $\lambda_1=0$ and $\lambda_i\leq\lambda_j$, for $i<j$ and $\lambda_2$ and $\lambda_N$ are, respectively, the smallest nonzero eigenvalue and maximum eigenvalue of $\lL$. Finally, given an edge $(i,j)$, $i$ is called a neighbor of $j$, and vice versa. We let $\vect{1}_{N}$ denote the vector of $N$ ones, and denote by $\vect{I}_{N}$ the $N\times N$ identity matrix. We also
define $\vect{\mathfrak{r}}=\frac{1}{\sqrt{N}}\vect{1}_{N}$,  $\vect{\mathfrak{R}}\in\mathbb{R}^{N\times (N-1)}$ and $\vectsf{T}=\begin{bmatrix}\vect{\mathfrak{r}}\quad\vect{\mathfrak{R}}\end{bmatrix}$, such that $\left[\begin{smallmatrix}\vect{\mathfrak{r}}&~\vectsf{\mathfrak{R}}\end{smallmatrix}\right]\left[\begin{smallmatrix}\vect{\mathfrak{r}}^\top\\\vect{\mathfrak{R}}^\top\end{smallmatrix}\right]=\left[\begin{smallmatrix}\vect{\mathfrak{r}}^\top\\\vect{\mathfrak{R}}^\top\end{smallmatrix}\right]\left[\begin{smallmatrix}\vect{\mathfrak{r}}&~\vectsf{\mathfrak{R}}\end{smallmatrix}\right]=\vect{I}_N$. Note that $\vectsf{T}^{\top}\vectsf{T}=\vectsf{T}\vectsf{T}^{\top}=\vect{I}$, and for a connected graph, $\vectsf{T}^{\top}\vectsf{L}\vectsf{T}=\begin{bmatrix}0 &\vect{0} \\ \vect{0} & \vectsf{L}^{+}\end{bmatrix}$, where $\vectsf{L}^{+}=\vect{\mathfrak{R}}^{\top}\vectsf{L}\vect{\mathfrak{R}}$. $\vectsf{L}^{+}$ is a positive definite matrix with eigenvalues $\{\lambda_i\}_{i=2}^{N}\in\mathbb{R}_{>0}$. For brevity and ease of presentation, the following notations are used alternatively: $\nabla_{\vect{x}^i\vect{x}^i}f^i_t\equiv\nabla_{\vect{x}^i\vect{x}^i}f^i(\vect{x}^i(t),t)$, $\nabla_{\vect{x}^i}f^i_t\equiv\nabla_{\vect{x}^i}f^i(\vect{x}^i(t),t)$ and $\nabla_{\vect{x}^i t}f^i_t\equiv\nabla_{\vect{x}^i t}f^i(\vect{x}^i(t),t)$.

\section{PROBLEM SETTING}
Our objective is to design an algorithm that drives local states towards the optimal trajectory, i.e., $\vect{x}^\star$ which is the solution of the unconstrained optimization problem \eqref{eq::problem}. The total cost is the summation of each agent $i$'s strongly-convex local costs $f^i$. In a fully decentralized setting, local costs are private information that are only available to the agent solely. Therefore, a distributed solution is presented in this paper to let agents track $\vect{x}^\star$. In this section, we provide the insights needed for proposing our novel algorithm. Let us first consider solving problem~\eqref{eq::problem} where the case is that the agents are aware of the local costs $f^i$, for all $i\in\mathcal{V}$. This can be done by implementing a central solver, e.g. using the prediction-correction method in~\cite{simonetto2016class}, where a descent direction drives the state $\vect{x}$ towards the optimal trajectory. Next, we discuss how this descent direction can be estimated distributively in a setting where agents are limited to communicate with their neighbors only. Inspired by the central solution in~\cite{simonetto2016class}, each agent can use the descent direction $\nabla_{\vect{x}\vect{x}}F^{-1}(\vect{x},t)(\nabla_{\vect{x}}F(\vect{x},t)+\nabla_{\vect{x}t}F(\vect{x},t))$ where $F(\vect{x}(t), t)=\frac{1}{N}\sum_{i=1}^{N}f^i(\vect{x}(t), t)$, to asymptotically converge to the solution. Therefore, by implementing the dynamics
\begin{align} \label{eq::decdircen}
    \dot{\vect{x}}^i(t)=&-\nabla_{\vect{x}\vect{x}}F^{-1}(\vect{x}^i(t), t)(\nabla_{\vect{x}}F(\vect{x}^i(t), t) \nonumber \\
    &+\nabla_{\vect{x}t}F(\vect{x}^i(t), t)),
\end{align}
all the states $\vect{x}^i$, for $i\in\VV$, converge to $\vect{x}^\star$ asymptotically. In a central manner, since all the states converge to the solution asymptotically, naturally, they also converge to a single trajectory. Therefore, consensus is achieved without any further manipulation. In the proposed method, we build an estimate of the global descent direction denoted as
\begin{align} \label{eq::dt}
    \vect{d}_t=(\sum_{i=1}^{N}\nabla_{\vect{x}^i\vect{x}^i}f^i_t)^{-1}(\sum_{i=1}^{N}\nabla_{\vect{x}^i}f^i_t+\nabla_{\vect{x}^i t}f^i_t),
\end{align}
by utilizing a weighted average consensus algorithm. However, \eqref{eq::dt} is different than the one in the dynamics~\eqref{eq::decdircen}. In the former, only the values of the first and second derivatives, calculated with local states as the input, are available to the agents, and in the latter, the total cost derivatives are available as functions to every agent and therefore, each agent can compute the exact local descent direction by using its local state as the input. Since the former is estimated in the proposed algorithm, agents converge to a single trajectory only in the case where initial conditions are similar globally. Therefore, we design a dynamics where a consensus term is added that reduces the difference between state values and minimize the total cost while reaching consensus in $\vect{x}^i$'s.

\section{MAIN RESULT} \label{sec::mainResult}
In this section, we introduce a novel algorithm that solves the distributed time-varying unconstrained optimization problem where the global cost is the summation of local cost functions. Following the previous section, agents of the network estimate the local descent direction by a weighted average consensus algorithm in discrete-time, and use this direction to converge to a neighborhood of the optimal solution while also trying to achieve consensus in their states. We propose the algorithm
\begin{subequations}\label{eq::mainAlg}
\begin{align}
    &\begin{cases}\label{eq::mainAlga}
    \vect{v}^i(k+1) = \vect{v}^i(k)+\delta_{c}\sum_{i=1}^{N}(\vect{p}^i(k)-\vect{p}^j(k)), \vspace{8pt}\\
    \vect{z}^i(k+1) = \vect{z}^i(k)-\delta_{c}\big(\vect{H}^i(k)\vect{p}^i(k)\!-\!\vect{g}^i(k)\! \vspace{5pt}\\
    -\vect{h}^i(k)\!+\!\sum_{i=1}^{N}(\vect{p}^i(k)\!-\!\vect{p}^j(k))\!+\!(\vect{v}^i(k)\!-\!\vect{v}^j(k))\big), \vspace{8pt}\\
    \vect{p}^i(k)=\vect{z}^i(k)+\vect{g}_s^i(k)+\vect{h}_s^i(k), \quad k\in\mathbb{Z}_{\geq0},
    \end{cases} \\
    &\dot{\vect{x}}^i(t)=-\vect{\psi}^i(t)-\sum_{i=1}^{N}(\vect{x}^i(t)-\vect{x}^j(t)), \quad t\geq0, \label{eq::mainAlgb} \\
    &\vect{v}^i(0), \vect{z}^i(0), \vect{x}^i(0)\in\mathbb{R}^{n},\quad i\in\VV. \nonumber
\end{align}
\end{subequations}
In this algorithm,
\begin{align*}
    &\vect{H}^i(k)=\nabla_{\vect{x}^i\vect{x}^i}f^i_{t_s}, \quad \vect{g}^i(k)=\nabla_{\vect{x}^i}f_{t_s}^i, \quad \vect{h}^i(k)=\nabla_{\vect{x}^i t}f_{t_s}^i, \\
    &\vect{\psi}^i(t)=\vect{p}^i(s\bar{k}),
\end{align*}
for $t\!\in\![t_{s},t_{s+1}), t_s\!=\!\delta_{t}s$ and $s\!=\!\{0,1,\cdots\}$ are switching signals that their roles is explained in the following. Here, \eqref{eq::mainAlga} constructs an estimate of the local descent direction by the state $\vect{p}^i(k)$ and~\eqref{eq::mainAlgb} drives the states using updates of $\vect{p}^i(k)$. These updates are passed to~\eqref{eq::mainAlgb} every $\bar{k}$ steps that~\eqref{eq::mainAlga} takes. While~\eqref{eq::mainAlga} takes $\bar{k}$ steps, \eqref{eq::mainAlgb} proceeds $\delta_s>0$ in time. Therefore, $\vect{\psi}^i(t)$ is a switching signal where at times $t_s$ is updated by $\vect{p}^i(s\bar{k})$, sampled from every $\bar{k}$ steps that~\eqref{eq::mainAlga} takes; $\vect{q}^i(t)$ is therefore constant in the time range $t\in[t_{s},t_{s+1})$ and is switched to the next value at each time instant $t_{s}$. Note that $\delta_{t}$ is the time span between the switchings of $\vect{\psi}^i(t)$. Moreover, at the same time instant $t=t_s$, the consensus Algorithm~\eqref{eq::mainAlga} updates its reference values $\vect{H}^i(k), \vect{g}^i(k)$ and $\vect{h}^i(k)$ using $\vect{x}^i(t=t_s)$ every $\bar{k}$ steps. With this mechanism, \eqref{eq::mainAlga} tracks the weighted average $(\sum_{i=1}^{N}\vect{H}^i(k))^{-1}(\sum_{i=1}^{N}\vect{g}^i(k)+\vect{h}^i(k))$ while~\eqref{eq::mainAlgb} uses the updates $\vect{p}^i(s\bar{k})$ from~\eqref{eq::mainAlga} to track the optimal trajectory $\vect{x}^{\star}(t)$.

The algorithm presented in~\eqref{eq::mainAlga} is inspired by a weighted average consensus algorithm in literature~\cite{YC-SSK:21} where by incorporating the local Hessian matrix $\vect{H}^i(k)$ as the weight and $\vect{g}^i(k)+\vect{h}^i(k)$ as the time-varying reference signal, each $\vect{p}^i$ converges to a neighborhood of the signal $(\sum_{i=1}^{N}\vect{H}^i(k))^{-1}(\sum_{i=1}^{N}\vect{g}^i(k)+\vect{h}^i(k))$. As noticed, this value is not exactly the one in~\eqref{eq::dt}; however, by considering some common assumptions, we characterize the error between the estimate and the actual value of $\vect{d}_t$ and show that by passing the updates $\vect{p}^i(s\bar{k})$ to~\eqref{eq::mainAlgb}, $\vect{x}^i$ converges to a neighborhood of the optimal trajectory. In order to prove convergence, some common conditions presented in e.g., \cite{simonetto2016class} and \cite{rahili2016distributed}, are required which are stated as following. The first assumption considers lower and upper bounds on the second derivative of the cost functions.
\begin{assumption} \label{ass::secDer}
Each local cost function $f^i(\vect{x}^i,t)$ is twice differentiable and uniformly in $t$. Also, $f^i$ is $m^i$-strongly convex and $l^i$-Lipschitz continuous, i.e.,
\begin{align*}
    m^i\vect{I}\leq\nabla_{\vect{x}^i\vect{x}^i}f^i_t\leq l^i\vect{I}, \quad \vect{x}^i\in\mathbb{R}^n, t\geq0.
\end{align*}
We also define $m=\textup{min}\{m^i\}$ and $l=\textup{max}\{l^i\}$ for $i\in\VV$.
\end{assumption}
The second assumption, considers bounds on the first derivatives of the cost functions.
\begin{assumption} \label{ass::firDer}
Local cost functions $f^i(\vect{x}^i,t)$ are sufficiently smooth in $\vect{x}^i$ and $t$, and the following bounds on the first derivatives of the local cost functions exist:
\begin{align*}
    ||\nabla_{\vect{x}^i}f^i_t||\leq C_0, \quad ||\nabla_{\vect{x}^i t}f^i_t||\leq C_1, \quad i\in\VV.
\end{align*}
where $C_0,C_1>0$.
\end{assumption}
By presenting the second assumption, it is also deduced that the variations of the first and second derivatives of the costs are bounded. Based on the requirements in Assumptions~\ref{ass::secDer} and~\ref{ass::firDer}, we can trivially calculate the bound on~\eqref{eq::dt} as $||\vect{d}_t||\leq C_d$, $t\geq0$, where $C
_d=\frac{1}{m}(C_0+C_1)$. This result is later used to characterize the bound of the tracking error.

Let us first examine the convergence of \eqref{eq::mainAlga} to a neighborhood of the local descent direction~\eqref{eq::dt}. We consider the following definitions for the proof of convergence. The weighted average to be tracked in~\eqref{eq::mainAlga} is $\bar{\vect{p}}(k)=(\sum_{i=1}^{N}\vect{H}^i(k))^{-1}(\sum_{i=1}^{N}\vect{g}^i(k)+\vect{h}^i(k))$ with its aggregated vector $\bar{\vectsf{p}}(k)=\bar{\vect{p}}(k)\otimes\vect{1}_N$. Again, $\bar{\vect{p}}(k)$ is updated every $\bar{k}$ steps which is equal to $\delta_t$ time in~\eqref{eq::mainAlgb}. Its variation over consecutive steps is $\Delta\bar{\vectsf{p}}(k)=\bar{\vectsf{p}}(k+1)-\bar{\vectsf{p}}(k)$. The diagonal matrix of Hessians is $\vectsf{H}(k)=\textup{diag}(\vect{H}^1(k),\cdots,\vect{H}^N(k))$. The gradient variations is defined as $\Delta\nabla\vectsf{f}(k)=\nabla\vectsf{f}(k+1)-\nabla\vectsf{f}(k)$ where $\nabla\vectsf{f}(k)$ is the aggregated vector of $\vect{g}^i(k)+\vect{h}^i(k)$ for $i\in\VV$; and finally, $\vectsf{w}(k)=\nabla\vectsf{f}(k)-\vectsf{H}(k)\bar{\vectsf{p}}(k)$ and $\Delta\vectsf{w}(k)=\vectsf{w}(k+1)-\vectsf{w}(k)$. Using the definitions above, the compact form of~\eqref{eq::mainAlga} is
\begin{subequations}\label{eq::comCons}
\begin{align}
    \vectsf{v}(k+1)&=\vectsf{v}(k)+\delta_c\vectsf{L}\vectsf{p}(k), \label{eq::comConsa} \\
    \vectsf{z}(k+1)&=\vectsf{z}(k)\!-\!\delta_c\big(\vectsf{H}(k)\vectsf{p}(k)\!-\!\nabla\vectsf{f}(k)\!+\!\vectsf{L}(\vectsf{p}(k)\!+\!\vectsf{v}(k)), \label{eq::comConsb} \\
    \vectsf{p}(k)&=\vectsf{z}(k)+\nabla\vectsf{f}(k). \label{eq::comConsc}
\end{align}
\end{subequations}
Using the change of variable $\bar{\vectsf{e}}=\vectsf{T}^{\top}(\vectsf{p}-\bar{\vectsf{p}})$ and $\begin{bmatrix}\vect{q}_1 \quad \vectsf{q}_{2:N}^{\top}\end{bmatrix}^{\top}=\vectsf{T}^{\top}(\vectsf{L}\vectsf{v}-\vectsf{w})$, \eqref{eq::comCons} is equivalent to
\begin{subequations} \label{eq::comConsTrans}
\begin{align}
    \vect{q}_1(k+1)&=\vect{q}_1(k), \\
    \footnotesize\begin{bmatrix}
    \bar{\vectsf{e}}(k+1) \\
    \vectsf{q}_{2:N}(k+1)
    \end{bmatrix}
    \footnotesize&\!=\!
    \footnotesize(\vect{I}\!+\!\delta_c\bar{\vectsf{A}}(k))
    \begin{bmatrix}
    \bar{\vectsf{e}}(k) \\
    \vectsf{q}_{2:N}(k)
    \end{bmatrix}\!+\!\bar{\vectsf{B}}
    \begin{bmatrix}
    \Delta\nabla\vectsf{f}(k)\!-\!\Delta\bar{\vectsf{p}}(k) \\
    \Delta\vectsf{w}(k)
    \end{bmatrix}
\end{align}
\end{subequations}
where $\footnotesize\bar{\vectsf{A}}(k)=\begin{bmatrix}&-\!\vectsf{T}^{\top}\!\otimes\!\vect{I}(\vectsf{H}(k)\!+\!\vectsf{L}\!\otimes\!\vect{I})\vectsf{T}\!\otimes\!\vect{I} \quad &-\!\begin{bmatrix}\vect{0} \\ \vect{I}_{n(N\!-\!1)}\end{bmatrix} \\ &\begin{bmatrix}\vect{0} \quad \vectsf{L}^{+}\vectsf{L}^{+}\!\otimes\vect{I} \end{bmatrix} \quad &\vect{0} \end{bmatrix}$ and $\footnotesize\bar{\vectsf{B}}~\!=\!~\begin{bmatrix}\vectsf{T}^{\top} \quad \vect{0} \\ \vect{0} \quad \vect{\mathfrak{R}}^{\top}\end{bmatrix}\otimes\vect{I}$. We now obtain the admissible step size $\delta_c$ to prove the internal stability of~\eqref{eq::comConsTrans}, i.e., the matrix $\vect{I}+\delta_c\bar{\vectsf{A}}(k)$ is Schur for $k=\{0,1,2,\cdots\}$.
\begin{lemma} \label{lem::delta}
    Under the Assumption~\eqref{ass::secDer} and by the virtue of the results in~\cite[Lemma 2]{YC-SSK:21} and~\cite[Lemma 3]{YC-SSK:21},
    if $\delta_c\in(0,\bar{\delta})$ in which $\bar{\delta}=\textup{min}\Big\{\{-2\frac{\textup{Re}(\gamma_{i,k})}{|\gamma_{i,k}|^2}\}_{i=1}^{2N-1}\Big\}_{k\in\mathbb{Z}_{\geq0}}$ where $\{\gamma_{i,k}\}_{i=1}^{2N-1}$ are the set of eigenvalues of $\bar{\vectsf{A}}(k)$, then every subsystem $\vect{I}+\delta_c\bar{\vectsf{A}}(k)$, $k\in\mathbb{Z}_{\geq0}$ is Schur. Moreover, we define $\phi=\textup{max}\{\|\vect{I}+\delta_c\bar{\vectsf{A}}(k)\|\}_{k\in\mathbb{Z}_{\geq0}}$, where we know that $\phi<1$.
\end{lemma}
The proof is present in~\cite{YC-SSK:21}. The result above, provided internal stability for the transformed algorithm in~\eqref{eq::comConsTrans}. We now seek a bound on the error between the trajectories of $\vect{p}^i(k)$ and the weighted average $\bar{\vect{p}}(k)$. Since the first derivatives of the local costs are bounded according to Assumption~\ref{ass::firDer}, $\vect{p}^i(k)$'s converge to a neighborhood of the weighted average with a maximum error characterized in the next result.
\begin{theorem} \label{thm::thm1}
    Let the agents of an undirected connected graph $\mathcal{G}$, implement~\eqref{eq::comCons} where the first gradients of local costs satisfy Assumption~\ref{ass::firDer}.
    Considering the result in Lemma~\ref{lem::delta}, i.e., $\delta_c\in(0,\bar{\delta})$, we have
    \begin{align} \label{eq::pbound}
        ||\vectsf{p}(k)-\bar{\vectsf{p}}(k)||\leq N\bar{C}\frac{1+\phi}{1-\phi^2}.
    \end{align}
    where $\bar{C}$ is defined in~\eqref{eq::Cbar}.
\end{theorem}
\begin{proof}
    To prove convergence of $\vectsf{p}(k)$ to a neighborhood of the weighted average $\bar{\vectsf{p}}(k)$, we use the Lyapunov stability analysis. Let us consider the transformed dynamics~\eqref{eq::comConsTrans} and define the Lyapunov function as a quadratic product of the states
    \begin{align*}
        V(k)=\begin{bmatrix}\bar{\vectsf{e}}(k) \\
    \vectsf{q}_{2:N}(k)
    \end{bmatrix}^{\top}\begin{bmatrix}\bar{\vectsf{e}}(k) \\
    \vectsf{q}_{2:N}(k)
    \end{bmatrix}.
    \end{align*}
    Here, we seek to prove that the variation of the Lyapunov function $V(k)$ at each step is negative and obtain a bound on the convergence error. Therefore, by defining $\Delta V(k)=V(k+1)-V(k)$, we have
    \begin{align*}
        \Delta V(k)&=\vectsf{y}^{\top}(k)(\vect{I}+\delta_c\bar{\vectsf{A}}(k))^{\top}(\vect{I}+\delta_c\bar{\vectsf{A}}(k))\vectsf{y}(k) \\
        &+2\vectsf{y}^{\top}(k)(\vect{I}\!+\!\delta_c\bar{\vectsf{A}}(k))^{\top}\bar{\vectsf{B}}\begin{bmatrix}
        \Delta\nabla\vectsf{f}(k)\!-\!\Delta\bar{\vectsf{p}}(k) \\
        \Delta\vectsf{w}(k)
        \end{bmatrix} \\
        &+\!\begin{bmatrix}
        \Delta\nabla\vectsf{f}(k)\!-\!\Delta\bar{\vectsf{p}}(k) \\
        \Delta\vectsf{w}(k)
        \end{bmatrix}^{\top}\bar{\vectsf{B}}\!^{\top}\bar{\vectsf{B}}\begin{bmatrix}
        \Delta\nabla\vectsf{f}(k)\!-\!\Delta\bar{\vectsf{p}}(k) \\
        \Delta\vectsf{w}(k)
        \end{bmatrix} \\
        &-\vectsf{y}^{\top}(k)\vectsf{y}(k),
    \end{align*}
    where $\vectsf{y}=\begin{bmatrix}\bar{\vectsf{e}} \\
    \vectsf{q}_{2:N}
    \end{bmatrix}$ is the aggregated vector of the states. Based on the results we have from Lemma~\ref{lem::delta}, by choosing the step size in the admissible range $\delta_c\in(0,\bar{\delta})$, then it is true that $\vect{I}+\delta_c\bar{\vectsf{A}}(k)$ is Schur. By incorporating the bounds $||\bar{\vectsf{B}}||<1$ and the ones in Assumptions~\ref{ass::secDer} and~\ref{ass::firDer} and their results, we can write
    \begin{align*}
        \Delta V(k)&\leq(\phi^2-1)\|\vectsf{y}(k)\|^2+2\phi N\bar{C}\|\vectsf{y}(k)\|+N^2\bar{C}^2,
    \end{align*}
    where
    \begin{align} \label{eq::Cbar}
        \bar{C}=8(C_0+C_1)+2(1+2l)C_d.
    \end{align}
    Since $\phi<1$, according to the inequality above, if the norm $||\vectsf{y}||$ is sufficiently large, the variation of the Lyapunov function becomes negative. In addition, if $\Delta V(k)<0$, then $||\vectsf{y}||$ decreases. Therefore, the value $||\vectsf{y}||$ is bounded. By using the Lyapunov stability analysis, we find that $||\vectsf{y}||\leq N\bar{C}\frac{1+\phi}{1-\phi^2}$. Given that $||\vectsf{p}(k)-\bar{\vectsf{p}}(k)||\leq\|\vectsf{y}(k)\|$, we can establish the tracking error in~\eqref{eq::pbound}.
\end{proof}
In the statement above, we obtained an upper bound on the error between the trajectories $\vectsf{p}(k)$ and the weighted average $\bar{\vectsf{p}}(k)$. Let us now consider $\vectsf{\Psi}=[\vect{\psi}^{1\top},\cdots,\vect{\psi}^{N\top}]^\top$ (aggregated vector of $\vect{\psi}^i$) which is samples of $\vect{p}^i(k)$ for every $\bar{k}$ steps. Before we establish the final result, we want to characterize a bound on the error between $\vectsf{\psi}(t)$ and $\vectsf{d}_t=\vect{d}_t\otimes\vect{1}_N$ denoted as
\begin{align} \label{eq::epsilon}
    \vectsf{\epsilon}_{t}=\vectsf{\Psi}(t)-\vectsf{d}_t.
\end{align}
Trivially, by the virtue of Theorem~\ref{thm::thm1} and the fact that $||\bar{\vectsf{p}}(k)-\vectsf{d}_{t}||\leq2NC_d$ (as a result of Assumption~\ref{ass::firDer}) for any $t\geq0$ and $k\in\mathbb{Z}_{\geq0}$, we can conclude that $||\vectsf{q}(t)-\vectsf{d}_{t}||=||\vectsf{\epsilon}_{t}||\leq N\bar{\epsilon}$ where $\bar{\epsilon}=2C_d+\bar{C}\frac{1+\phi}{1-\phi^2}$. Let us now present the final statement.
\begin{theorem} \label{thm::thm2}
    Let the agents of an undirected connected graph $\mathcal{G}$ implement the Algorithm~\eqref{eq::mainAlg} to track the optimal trajectory $\vect{x}^{\star}(t)$, the solution of the unconstrained optimization problem~\eqref{eq::problem}. Provided that $\delta_c\in(0,\bar{\delta})$, we can prove that the gradient of the total cost asymptotically converges to a neighborhood of the origin with the bound
    \begin{align} \label{eq::gradientBound}
        ||\frac{1}{N}\sum_{i=1}^{N}\nabla_{\vect{x}^i}f^i(\vect{x}^i(t),t)-\frac{1}{N}\sum_{i=1}^{N}\nabla_{\vect{x}}f^i(\vect{x}^{\star}(t),t)||\leq C_{\nabla}\bar{\epsilon}
    \end{align}
    where $C_{\nabla}$ is defined in~\eqref{eq::Cnabla}.
\end{theorem}
\begin{proof}
   To simplify the presentation, we demonstrate the proof for when the cost functions are univariant, i.e., the decision variable $\vect{x}^i$ is scalar and therefore, $x^i\in\mathbb{R}$, $i\in\VV$. Also, $\vectsf{H}_t$, $\nabla_{\vectsf{x}}\vectsf{f}_t$ and $\nabla_{\vectsf{x} t}\vectsf{f}_t$ are aggregated matrix and vectors of the local Hessians and first derivatives of costs $f^i(\vect{x}^i(t),t)$. We implement the results from Theorem~\ref{thm::thm1} and use a Lyapunov stability analysis in the continuous-time framework to prove convergence. Consider the dynamics~\eqref{eq::mainAlgb} which can be presented in the compact form
   \begin{align} \label{eq::dynCom}
       \dot{\vectsf{x}}(t)=-\vectsf{\Psi}(t)-\vectsf{L}\vectsf{x}(t).
   \end{align}
   Let us define the Lyapunov function
   \begin{align*}
       V(\vectsf{x}(t),t)=(\vect{1}_N^{\top}\nabla_{\vectsf{x}}\vectsf{f}_t)^2+\alpha\vectsf{x}^{\top}(t)\vectsf{L}\vectsf{x}(t),
   \end{align*}
   where $\alpha>0$ is a positive scalar; for the rest of the proof, we use $\vectsf{x}_t$ and $V_t$ as replacements of $\vectsf{x}(t)$ and $V(\vectsf{x}(t),t)$, respectively. Under the assumption~\eqref{ass::firDer}, the Lyapunov function is bounded by the states $\vectsf{x}$ in the dynamics~\eqref{eq::dynCom} by $0\leq V_t\leq C_1^2+\alpha\lambda_N||\vectsf{x}_t||^2$. Taking the derivative of the Lyapunov function, we have
   \begin{align*}
       \dot{V}_t=\nabla_{\vectsf{x}}\vectsf{f}_t^\top\vect{1}_N\vect{1}_N^{\top}(\vectsf{H}_t\dot{\vectsf{x}}_t+\nabla_{\vectsf{x}t}\vectsf{f}_t)+\alpha\vectsf{x}_t^{\top}\vectsf{L}\dot{\vectsf{x}}_t,
   \end{align*}
   where by substituting $\dot{\vectsf{x}}_t$, $\vectsf{\Psi}$ and $\vect{d}_t$ by the equations~\eqref{eq::dynCom}, \eqref{eq::epsilon} and \eqref{eq::dt}, respectively, we get
   \begin{align*}
       \dot{V}_t&=-\nabla_{\vectsf{x}}\vectsf{f}_t^{\top}\vect{1}_N^{\top}\vect{1}_N\nabla_{\vectsf{x}}\vectsf{f}_t-\nabla_{\vectsf{x}}\vectsf{f}_t^{\top}\vect{1}_N\vect{1}_N^{\top}\vectsf{H}_t\vectsf{\epsilon}_{t} \\
       &-\nabla_{\vectsf{x}}\vectsf{f}_t^{\top}\vect{1}_N\vect{1}_N^{\top}\vectsf{H}_t\vectsf{L}\vectsf{x}_t-\alpha\vectsf{x}_t^{\top}\vectsf{L}\vectsf{\epsilon}_{t}-\alpha\vectsf{x}_t^{\top}\vectsf{L}\vectsf{x}_t.
   \end{align*}
   By subtraction and addition of similar terms, the equality above can be rewritten as
   \begin{align*}
       \dot{V}_t&=-\nabla_{\vectsf{x}}\vectsf{f}_t^{\top}\vect{1}_N^{\top}\vect{1}_N\nabla_{\vectsf{x}}\vectsf{f}_t
       -\nabla_{\vectsf{x}}\vectsf{f}_t^{\top}\vect{1}_N\vect{1}_N^{\top}\vectsf{H}_t\vectsf{\epsilon}_{t} \\
       &-\frac{1}{2}||\frac{1}{\sqrt{\beta}}\vectsf{H}_t\vect{1}_N\vect{1}_N^{\top}\nabla_{\vectsf{x}}\vectsf{f}_t+\sqrt{\beta}\vectsf{L}\vectsf{x}_t||^2 \\
       &+\frac{1}{2\beta}\nabla_{\vectsf{x}}\vectsf{f}_t^{\top}\vect{1}_N\vect{1}_N^{\top}\vectsf{H}_t^{2}\vect{1}_N\vect{1}_N^{\top}\nabla_{\vectsf{x}}\vectsf{f}_t+\frac{\beta}{2}\vectsf{x}_t^{\top}\vectsf{L}^{2}\vectsf{x}_t \\
       &-\frac{\alpha}{2}||\frac{1}{\sqrt{\gamma}}\vectsf{\epsilon}_{t}+\sqrt{\gamma}\vectsf{L}\vectsf{x}_t||^2+\frac{\alpha}{2\gamma}\vectsf{\epsilon}_{t}^{\top}\vectsf{\epsilon}_{t}+\frac{\alpha\gamma}{2}\vectsf{x}_t^{\top}\vectsf{L}^{2}\vectsf{x}_t \\
       &-\alpha\vectsf{x}_t^{\top}\vectsf{L}^{2}\vectsf{x}_t,
   \end{align*}
   where $\beta,\gamma>0$. Using the Assumptions~\eqref{ass::secDer} and~\eqref{ass::firDer}, we can derive the inequality
   \begin{align*}
       \dot{V}_t&\leq-(1-\frac{l^2}{2\beta})||\vect{1}_N^{\top}\nabla_{\vectsf{x}}\vectsf{f}_t^{\top}||^2
       +lN\bar{\epsilon}||\vect{1}_N^{\top}\nabla_{\vectsf{x}}\vectsf{f}_t^{\top}||+\frac{\alpha N^2}{2\gamma}\bar{\epsilon}^2 \\
       &-(\alpha-\frac{\alpha\gamma+\beta}{2})\vectsf{x}_t^{\top}\vectsf{L}\vectsf{x}_t-\frac{\alpha}{2}||\frac{1}{\sqrt{\gamma}}\vectsf{\epsilon}_{t}+\sqrt{\gamma}\vectsf{L}\vectsf{x}_t||^2 \\
       &-\frac{1}{2}||\frac{1}{\sqrt{\beta}}\vectsf{H}_t\vect{1}_N\vect{1}_N^{\top}\nabla_{\vectsf{x}}\vectsf{f}_t+\sqrt{\beta}\vectsf{L}\vectsf{x}_t||^2,
   \end{align*}
   in which, for stability, we have the requirements $\alpha<\frac{l^2}{4},\beta>\frac{l^2}{2}, \gamma<2-\frac{l^2}{2\alpha}$ which is always feasible. Based on the inequality above, we can conclude that the summation of the gradients, asymptotically converges to a neighborhood of the origin with the bound
   \begin{align*}
       ||\frac{1}{N}\vect{1}_N^{\top}\nabla_{\vectsf{x}}\vectsf{f}_t^{\top}||\leq\frac{l+\sqrt{l^2+\frac{\alpha}{\gamma}(2-\frac{l^2}{\beta})}}{2-\frac{l^2}{\beta}}\bar{\epsilon},
   \end{align*}
   where by defining
   \begin{align} \label{eq::Cnabla}
       C_{\nabla}=\frac{l+\sqrt{l^2+\frac{\alpha}{\gamma}(2-\frac{l^2}{\beta})}}{2-\frac{l^2}{\beta}},
   \end{align}
   and using the fact that $\frac{1}{N}\sum_{i=1}^{N}\nabla_{\vect{x}}f^i(\vect{x}^{\star}(t),t)=0$ we get the final result in~\eqref{eq::gradientBound}.
\end{proof}

\section{NUMERICAL EXAMPLE}
\begin{figure}[t]
    \centering
    {\includegraphics[trim= 1pt 5pt 16pt 5pt ,clip,width=.6\linewidth]{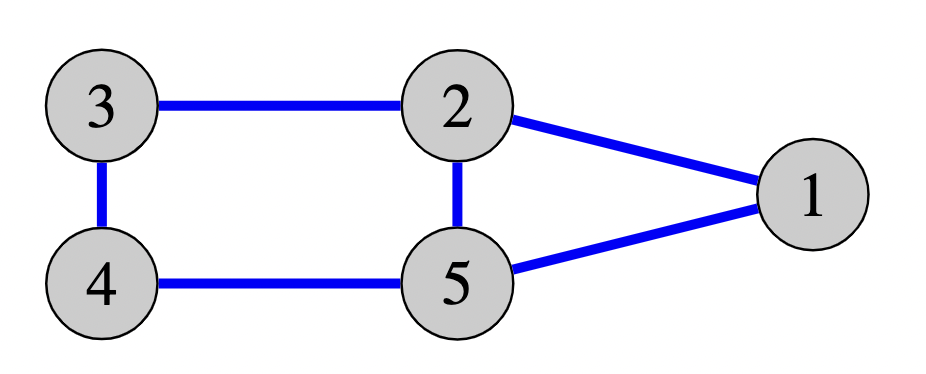}}
    \caption{\small{The graph indicated above is an undirected connected graph with adjacency weights of $\vectsf{a}_{ij}=1$, if $(i,j)\in\mathcal{E}$, otherwise $\vectsf{a}_{ij}=0$. Each agent $i\in\{1,2,\cdots,5\}$ is endowed with a local cost $f^i(x^i(t),t)=i(x^i(t))^2+\textup{sin}(i\omega t)x^i(t)$.}}
    \label{fig::fig0}
\end{figure}
To demonstrate the performance of the proposed method, we study convergence of Algorithm~\eqref{eq::mainAlg} in presence of local time-varying costs. Many problems such as ones in economic dispatch or linear regression are formulated as optimization problems with quadratic costs of the form $f(x(t),t)=\frac{1}{2}a(t)(x(t))^2+b(t)x(t)+c(t)$. In many cases, for example, due to parameter oscillations of local generators in economic dispatch or updates in local data sets in linear regression, the problem at hand is time-varying, and consequently, agents are required to track a time-variant solution rather than converging to a single minimum point. Therefore, we consider solving an unconstrained quadratic optimization problem in the following of this section.

Let a network of $N=5$ agents interact with each other to solve problem~\eqref{eq::problem}. The topology of the network is an undirected connected graph illustrated in Fig.~\ref{fig::fig0}. Each agent is endowed with a local cost
\begin{align*}
    f^i(x^i(t),t)=i(x^i(t))^2+\textup{sin}(i\omega t)x^i(t),
\end{align*}
where $\omega=0.05$ controls the frequency of the time-varying costs. Trivially, one can solve the problem $x^{\star}(t)=\textup{argmin}_{x}\frac{1}{N}\sum_{i=1}^{5}i(x(t))^2+\textup{sin}(i\omega t)x(t)$ analytically and derive the optimal solution as the time-varying trajectory $x^{\star}(t)=-\frac{1}{2}\frac{\sum_{i=1}^{5}\textup{sin}(i\omega t)}{\sum_{i=1}^{5}i}$. The objective is to implement Algorithm~\eqref{eq::mainAlg} to track $x^{\star}(t)$ with different values of $\bar{k}$ and observe its effect on the convergence error. In this example, $\delta_t=0.1$ is fixed. We set $\bar{k}=\{1,2,5,10\}$ and measure the state values $x^i(t)$ and the tracking error, defined as $e(t)=||\vectsf{x}(t)-x^{\star}(t)\vect{1}_5||_2$, in the range $t\in[0,T]$ with $T=50$. Moreover, to review the overall effect on convergence, a separate figure is plotted to show the average of the tracking error $\bar{e}=\frac{1}{T}\int_{\tau=0}^{\tau=T} e(\tau)$ over the time span $t\in[0,T]$, for each case of $\bar{k}$.

By plotting the states $x^i(t)$ over $t\in[0,T]$, for $i\in\{1,2,\cdots,5\}$, we can observe that convergence to the optimal trajectory has improved. While using $\bar{k}=1$ the agents can still track the optimal solution, increasing $\bar{k}$ to $5$ and $10$ results in more consensus between the states at each time instant $t$. Evidently, we can see from Fig.\ref{fig::fig2}(a) that convergence error is reduced when $\bar{k}=10$ compared to when $\bar{k}=\{1,2\}$. Intuitively, as $\bar{k}$ grows, agents obtain a more accurate estimate of the current descent direction $\vect{d}^i_t$ and drive their local states towards a tighter neighborhood of the optimal trajectory. As a result, Fig.~\ref{fig::fig2}(b) indicates that the average tracking error $\bar{e}$ is lower in cases with higher values of $\bar{k}$.

\begin{figure}[!t]
    \centering
    {\includegraphics[trim= 1pt 5pt 16pt 5pt ,clip,width=\linewidth]{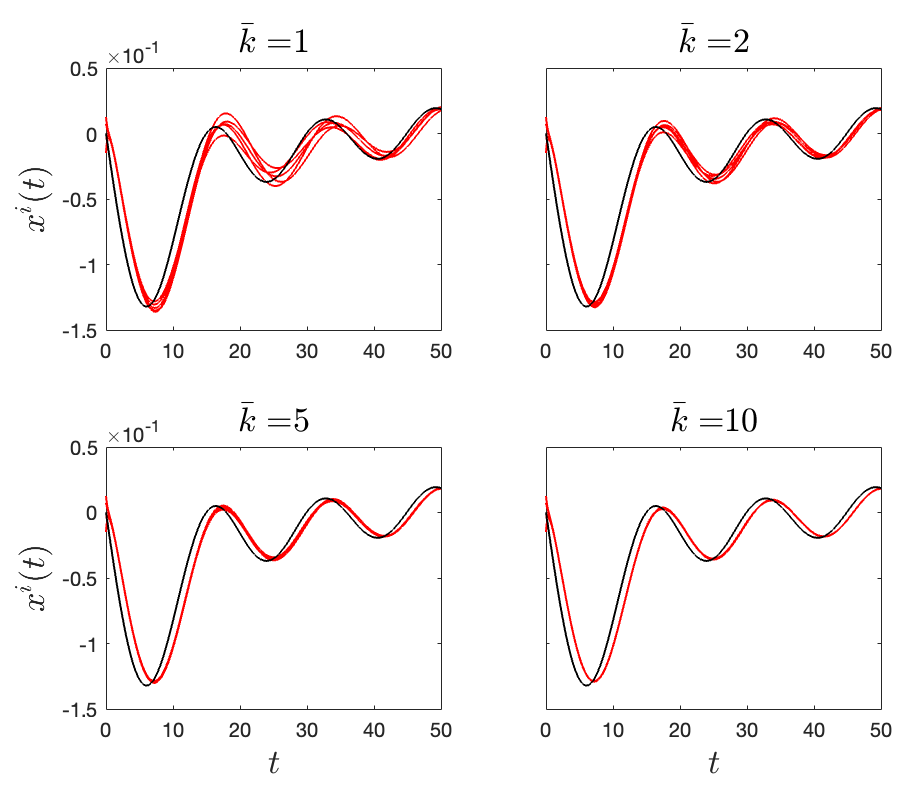}}
    \caption{\small{Trajectories of the states $x^i(t)$ for $i\in\{1,2,\cdots,5\}$ are shown above for different values of $\bar{k}$. As $\bar{k}$ increases, there is more consensus between the states $x^i(t)$ and the optimal trajectory is tracked more accurately.}}
    \label{fig::fig1}
\end{figure}

\begin{figure}[!t]
    \centering
    {\includegraphics[trim= 1pt 5pt 16pt 5pt ,clip,width=\linewidth]{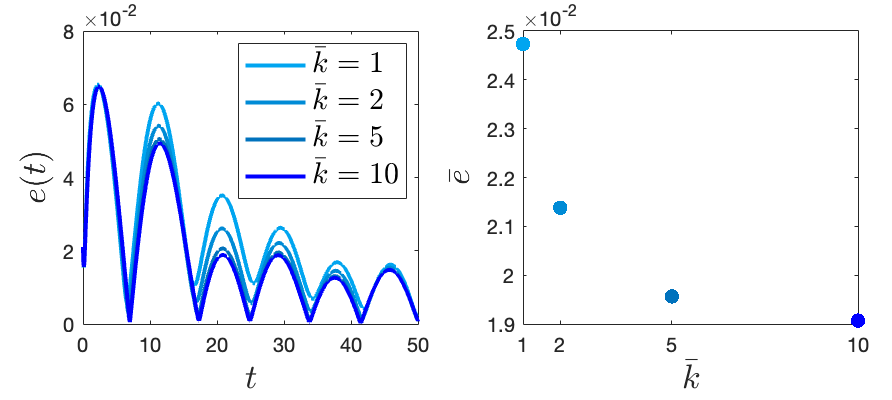}}
    \caption{\small{(a) Left figure: Represents the tracking error in the time interval $t\in[0,T]$ for $\bar{k}=\{1,2,5,10\}$. The error oscillates due to the nature of the local cost functions, however, when $\bar{k}$ is higher, there is less tracking error at each time instant. (b) Right figure: Represents the average tracking error $\bar{e}$ for $\bar{k}=\{1,2,5,10\}$. A similar conclusion is depicted here.}}
    \label{fig::fig2}
\end{figure}

\section{CONCLUSION}
We proposed a method to solve the distributed unconstrained optimization problem. In this setting, the total cost to be optimized consists of time-varying local costs that each agent of a network is endowed with and therefore, the solution is an optimal trajectory rather than a minimum point. In our approach, we implemented the discrete-time version of a weighted average consensus algorithm to derive an estimate of the descent direction, and constructed a continuous-time dynamics where this estimate was used to drive local states towards the optimal trajectory while reaching consensus. Under some common assumptions and with the use of the Lyapunov stability analysis, a bound on the asymptotic tracking error of the total cost gradient was achieved. To show the effect of the parameters used in the algorithm, a numerical example was provided where convergence to the optimal trajectory was studied with different values of these parameters.

\bibliographystyle{ieeetr}

\end{document}